\documentclass{article}
\usepackage[left=4cm, right=4cm, top=4cm]{geometry} 
\usepackage{amsmath,amsthm,amssymb,mathrsfs,amstext, titlesec}
\makeatletter

\titleformat{\section}
{\normalfont\Large\bfseries}{\thesection}{1em}{}
\titleformat*{\subsection}{\Large\bfseries}
\titleformat*{\subsubsection}{\large\bfseries}
\titleformat*{\paragraph}{\large\bfseries}
\titleformat*{\subparagraph}{\large\bfseries}

\renewcommand{\@seccntformat}[1]{\csname the#1\endcsname. }
\makeatother
\theoremstyle{plain}
\newtheorem{theorem}{Theorem}[section]

\newtheorem{lemma}[theorem]{Lemma}
\theoremstyle{definition}
\newtheorem{definition}[theorem]{Definition}
\newtheorem{observation}[theorem]{Observation}

\providecommand{\keywords}[1]{{\bf{Keywords:}} #1}
\providecommand{\subjectclass}[2]{\textbf{Mathematics subject classification:} #1}

\title{Central Sets Theorem along filters and some combinatorial
consequences
\footnote{\keywords{Ramsey Theory, The Central Sets Theorem, Algebra of the Stone-\v{C}ech compactification}}}

\date{}
\author{Sayan Goswami
\footnote{Department of Mathematics, 
           University of Kalyani, 
           Kalyani-741235,
           Nadia, West Bengal, India
           {\bf sayan92m@gmail.com}, {\bf sgmathku17@klyuniv.ac.in}}
           \and
Jyotirmoy Poddar
\footnote{Department of Mathematics, 
          University of Kalyani, 
          Kalyani-741235,
          Nadia, West Bengal, India
          {\bf jyotirmoypoddar1@gmail.com}}
}

\begin{document}
\maketitle

\begin{abstract}
\noindent The Central Sets Theorem was introduced by H. Furstenberg
and then afterwards several mathematicians have provided various versions
and extensions of this theorem. All of these theorems deal with central
sets, and its origin from the algebra of Stone-\v{C}ech compactification
of arbitrary semigroup, say $\beta S$. It can be proved that every
closed subsemigroup of $\beta S$ is generated by a filter. We will
show that, under some restrictions, one can derive the Central Sets
Theorem for any closed subsemigroup of $\beta S$. We will derive
this theorem using the corresponding filter and its algebra. Later
we will also deal with how the notions of largeness along filters
are preserved under some well behaved homomorphisms and give some
consequences.
\end{abstract}
\subjectclass{05D10}

\tableofcontents

\section{Introduction}

The study of large sets in Ramsey theory has a long history. It uses
special techniques from topological dynamics, ergodic theory, the
algebra of Stone-\v{C}ech compactification, combinatorics etc. Normally,
in Ramsey theory, we search for rich combinatorial structures contained
in various large sets. In \cite[Proposition 8.21]{key-28} , H. Furstenberg defined
the Central sets using the notions from topological dynamics, and proved
the Central Sets Theorem dynamically. Central Sets Theorem, basically
is the joint extension of the two famous theorems, van der Waerden's
theorem \cite{key-65}, regarding monochromatic configurations of arithmetic
progressions of some finite length, and Hindman's theorem \cite{key-35},
regarding the monochromatic configurations of finite sums of an infinite
sequence. Before proceeding further, we mention that throughout this article, we will use the following notation.
\begin{itemize}
\item $\mathcal{P}_{f}\left(X\right)$ denotes the family of all non empty finite
subsets of a set $X$.
\item For any semigroup $(S,+)$, and $t\in S$ and $A\subseteq S$, $t^{-1}A={s\in S: ts\in A}$.
\item For any $A,B\subseteq S$ let us denote $B^{-1}A=\bigcup_{t\in B}t^{-1}A$.
\end{itemize}
The original Central Sets Theorem is stated
below.
\begin{theorem} [The original Central Sets Theorem]
Let $A$ be a central subset of
$\mathbb{N}$, let $k\in\mathbb{N}$, and for each $i\in\left\{ 1,2,\ldots,k\right\} $,
let $\langle y_{i,n}\rangle_{n=1}^{\infty}$be a sequence in $\mathbb{Z}$.
There exist sequences $\langle a_{n}\rangle_{n=1}^{\infty}$ in $\mathbb{N}$
and $\langle H_{n}\rangle_{n=1}^{\infty}$ in $\mathcal{P}_{f}\left(\mathbb{N}\right)$
such that
   \begin{enumerate}
\item For each $n\in\mathbb{N}$, $\max H_{n}<\min H_{n+1}$ , and
\item For each $i\in\left\{ 1,2,\ldots,k\right\} $, and $F\in\mathcal{P}_{f}\left(\mathbb{N}\right)$,
we have
\[
\sum_{n\in F}\left(a_{n}+\sum_{t\in H_{n}}y_{i,t}\right)\in A.
\]
  \end{enumerate}
\end{theorem}

\begin{proof}
See \cite{key-28}.
\end{proof}
This theorem did point out some immediate consequences towards the
study of large sets and their abundance in Ramsey theory.

We now give a brief review about the Stone-\v{C}ech compactification
of a discrete semigroup. Let $\left(S,\cdot\right)$ be any discrete
semigroup and denote its Stone-\v{C}ech compactification by $\beta S$.
$\beta S$ is the set of all ultrafilters on $S$, where the points
of $S$ are identified with the principal ultrafilters. The basis
for the topology is $\left\{ \bar{A}:A\subseteq S\right\} $, where
$\bar{A}=\left\{ p\in\beta S:A\in p\right\} $. The operation of $S$
can be extended to $\beta S$ making $\left(\beta S,\cdot\right)$
a compact, right topological semigroup with $S$ contained in its
topological center. That is, for all $p\in\beta S$, the function
$\rho_{p}:\beta S\rightarrow\beta S$ is continuous, where $\rho_{p}\left(q\right)=q\cdot p$
and for all $x\in S$, the function $\lambda_{x}:\beta S\rightarrow\beta S$
is continuous, where $\lambda_{x}\left(q\right)=x\cdot q$. For $p,q\in\beta S,$
and $A\subseteq S$, $A\in p\cdot q$ if and only if $\left\{ x\in S:x^{-1}A\in q\right\} \in p$,
where $x^{-1}A=\left\{ y\in S:x\cdot y\in A\right\} $. 

Since $\beta S$ is a compact Hausdorff right topological semigroup,
it has a smallest two sided ideal denoted by $K\left(\beta S\right)$,
which is the union of all of the minimal right ideals of $S$, as
well as the union of all of the minimal left ideals of $S$. Every
left ideal of $\beta S$ contains a minimal left ideal, and every right
ideal of $\beta S$ contains a minimal right ideal. The intersection
of any minimal left ideal and any minimal right ideal is a group,
and any two such groups are isomorphic. Any idempotent $p$ in $\beta S$
is said to be minimal if and only if $p\in K\left(\beta S\right)$.
Though central sets was defined dynamically, there is an algebraic
counterpart of this definition, established by V. Bergelson and
N. Hindman in \cite{key-6}.
\begin{definition}
Let $S$ be a discrete semigroup. Then a subset $A$ of $S$ is called
central if and only if there is some minimal idempotent $p$ such
that $A\in p$.
\end{definition}

Due to the above characterization of Central sets, one can use the
algebra of the Stone-\v{C}ech compactification of a discrete semigroup.
Thereafter many versions of Central Sets Theorem came, among which
the following one is the general one for any arbitrary semigroup established in \cite{key-100}.
\begin{theorem}
Let $\left(S,\cdot\right)$ be a semigroup and let $C$ be a central
subset of $S$. There exist $m:\mathcal{P}_{f}\left(^{\mathbb{N}}S\right)\rightarrow\mathbb{N}$,
$\alpha\in\times_{F\in\mathcal{P}_{f}\left(^{\mathbb{N}}S\right)}S^{m(F)+1}$,
and $\tau\in\times_{F\in\mathcal{P}_{f}\left(^{\mathbb{N}}S\right)}\mathcal{J}_{m(F)}$
such that 

(1) If $F,G\in\mathcal{P}_{f}\left(^{\mathbb{N}}S\right)$ and $G\subsetneq F$,
then $\tau\left(G\right)\left(m\left(G\right)\right)$ $<$ $\tau\left(F\right)\left(1\right)$
and

(2) Whenever $n\in\mathbb{N}$, $G_{1},G_{2},\ldots,G_{n}\in\mathcal{P}_{f}\left(^{\mathbb{N}}S\right)$,
$G_{1}\subsetneq G_{2}\subsetneq\ldots\subsetneq G_{n}$, and for
each $i\in\left\{ 1,2,\cdots,n\right\} $, $f_{i}\in G_{i}$, one
has 
\[
\prod_{i=1}^{n}\left(\left(\prod_{j=1}^{m(G_{i})}\alpha\left(G_{i}\right)\left(j\right)\cdot f_{i}\left(\tau\left(G_{i}\right)\left(j\right)\right)\right)\cdot\alpha\left(G_{i}\right)\left(m\left(G_{i}\right)+1\right)\right)\in A.
\]
\end{theorem}

\begin{proof}
See \cite{key-100}.
\end{proof}
For a simpler version one can see \cite{key-3}, a general version one can see \cite{key-101} and for more details on Central Sets Theorem till date, see \cite{key-2}. For the sake
of our work we will need to revisit some important definitions arising
from the algebra of the Stone-\v{C}ech compactification of a discrete
semigroup, for more details see \cite{key-53}.

\begin{definition}
If $\left(S,\cdot\right)$ be a semigroup and $A\subseteq S$, then

  \begin{enumerate}
\item The set $A$ is thick if and only if for any finite subset $F$ of
$S$, there exists an element $x\in S$ such that $F\cdot x\subset A$.
This means the sets which contain a translation of any finite subset.
For example one can see $\cup_{n\in\mathbb{N}}\left[2^{n},2^{n}+n\right]$
is a thick set in $\mathbb{N}$.
\item The set $A$ is syndetic if and only if there exists a finite subset
$G$ of $S$ such that $\bigcup_{t\in G}t^{-1}A=S$. That is, with
a finite translation if, the set covers the entire semigroup,
then it will be called a Syndetic set. For example the set of even
and odd numbers are both syndetic in $\mathbb{N}$.

Note that $Gt\cap A\neq \emptyset$ if and only if $t\in G^{-1}A$. So $Gt\cap A\neq \emptyset$ if and only if $t\in S=G^{-1}A$. Hence the syndetic sets are dual to thick.
\item The sets which can be written as an intersection of a syndetic and
a thick set are called $\mathit{Piecewise}$ $\mathit{syndetic}$
sets. More formally a set $A$ is $\mathit{Piecewise}$ $\mathit{syndetic}$
if and only if there exists $G\in\mathcal{P}_{f}\left(S\right)$ such
that for every $F\in\mathcal{P}_{f}\left(S\right)$, there exists
$x\in S$ such that $F\cdot x\subseteq\bigcup_{t\in G}t^{-1}A$. Clearly
the thick sets and syndetic sets are natural examples of $\mathit{Piecewise}$
$\mathit{syndetic}$ sets. From definition one can immediately see
that $2\mathbb{N}\cap\bigcup_{n\in\mathbb{N}}\left[2^{n},2^{n}+n\right]$
is a nontrivial example of $\mathit{Piecewise}$ $\mathit{syndetic}$
sets in $\mathbb{N}$.
\item We define  $\mathcal{T}=\,^{\mathbb{N}}S$.
\item For $m\in\mathbb{N}$, $\mathcal{J}_{m}=\left\{ \left(t\left(1\right),\ldots,t\left(m\right)\right)\in\mathbb{N}^{m}:t\left(1\right)<\ldots<t\left(m\right)\right\} .$
\item Given $m\in\mathbb{N}$, $a\in S^{m+1}$, $t\in\mathcal{J}_{m}$ and
$f\in F$, 
\[
x\left(m,a,t,f\right)=\left(\prod_{j=1}^{m}\left(a\left(j\right)\cdot f\left(t\left(j\right)\right)\right)\right)\cdot a\left(m+1\right),
\]
where the terms in the product $\prod$ are arranged in increasing
order.
\item $A\subseteq S$ is called a $J$-set iff for each $F\in\mathcal{P}_{f}\left(\mathcal{T}\right)$,
there exists $m\in\mathbb{N}$, $a\in S^{m+1}$, $t\in\mathcal{J}_{m}$
such that, for each $f\in\mathcal{T}$,
\[
x\left(m,a,t,f\right)\in A.
\]
\item If the semigroup $S$ is commutative, the definition is rather simple.
In that case, a set $A\subseteq S$ is a $J$-set if and only if whenever
$F\in\mathcal{P}_{f}\left(^{\mathbb{N}}S\right)$, there exist $a\in S$
and $H\in\mathcal{P}_{f}\left(\mathbb{N}\right)$, such that for each
$f\in F$, $a+\sum_{t\in H}f(t)\in A$.
   \end{enumerate}
\end{definition}   
IP-sets play a major role in Ramsey theory. The importance
of this set was first shown by N. Hindman in \cite{key-35}. Let us
recall the following definitions around IP-sets.
\begin{definition}
If we are given any injective sequence $\langle x_{n}\rangle_{n=1}^{\infty}$
in $S$, then,

  \begin{enumerate}
\item A set $A$ which contains $FP\left(\langle x_{n}\rangle_{n=1}^{\infty}\right)$
for some injective sequence $\langle x_{n}\rangle_{n=1}^{\infty}$
in $S$, is called an IP set, where 
\[
FP\left(\langle x_{n}\rangle_{n=1}^{\infty}\right)=\left\{ x_{i_{1}}\cdot x_{i_{2}}\cdot\cdots\cdot x_{i_{n}}:\left\{ i_{1}<i_{2}<\cdots<i_{n}\right\} \subseteq\mathbb{N}\right\} .
\]
\item For any $r\in\mathbb{N}$, a set $A$ which contains $FP\left(\langle x_{n}\rangle_{n=1}^{r}\right)$
for some injective sequence $\langle x_{n}\rangle_{n=1}^{\infty}$
in $S$, is called an IP$_{r}$ set, where 
\[
FP\left(\langle x_{n}\rangle_{n=1}^{r}\right)=\left\{ x_{i_{1}}\cdot x_{i_{2}}\cdot\cdots\cdot x_{i_{n}}:\left\{ i_{1}<i_{2}<\cdots<i_{n}\right\} \subseteq\left\{ 1,2,\ldots,r\right\} \right\} .
\]
  \end{enumerate}
\end{definition}  
An easy algebraic characterization shows that a set $A$ is an IP-set,
if and only if $A\in p$, for some idempotent $p\in\beta S$. A set
which intersects every IP-set is called an IP$^{*}$ set. And hence a
set $A$ is called an IP$^{*}$set if and only if $A\in p$ for every
idempotent $p\in\beta S$. A set which intersects every IP$_{r}$
set is called an IP$_{r}^{*}$ set. Clearly an IP$_{r}^{*}$ set is
an IP$^{*}$-set, but the converse is not true.

We now discuss a different notion. Throughout this paper, $\mathcal{F}$
will denote a filter of $\left(S,\cdot\right)$. For every filter
$\mathcal{F}$ of $S$, define $\bar{\mathcal{F}}\subseteq\beta S$,
by
\[
\bar{\mathcal{F}}=\bigcap_{V\in\mathcal{F}}\bar{V}.
\]
 It is a routine check that $\bar{\mathcal{F}}$ is a closed subset
of $\beta S$ consisting of ultrafilters which contain $\mathcal{F}$.
If $\mathcal{F}$ is an idempotent filter, i.e., $\mathcal{F}\subset\mathcal{F}\cdot\mathcal{F}$,
then $\bar{\mathcal{F}}$ becomes a closed subsemigroup of $\beta S$,
but the converse is not true. Throughout our article, we will consider
only those filters $\mathcal{F}$, for which $\bar{\mathcal{F}}$
is a closed subsemigroup of $\beta S$. Note that if we take the filter
$\mathcal{F}=\left\{ S\right\} $, then $\bar{\mathcal{F}}=\beta S$.

In light of this notion we can define the concept of piecewise $\mathcal{F}$-syndeticity
both combinatorially and algebraically. For details see \cite{key-4}.
\begin{definition}
Let $\mathcal{F}$ be a filter on $S$. A subset $A\subseteq S$ is
piecewise $\mathcal{F}$-syndetic if for every $V\in\mathcal{F}$,
there is a finite $F_{V}\subseteq V$ and $W_{V}\in\mathcal{F}$ such
that whenever $H\subseteq W_{V}$ a finite subset, there is $y\in V$
such that $H\cdot y\subseteq F_{V}^{-1}A$.
\end{definition}

Here is an algebraic characterization of piecewise $\mathcal{F}$-syndetic
sets.
\begin{theorem}
Let $T$ be a closed subsemigroup of $\beta S$, and $\mathcal{F}$ be
the filter on $S$ such that $T=\bar{\mathcal{F}}$, also let $A\subseteq S$.
Then $\bar{A}\cap K\left(T\right)\neq\emptyset$ (Here $K(T)$ is the smallest two sided ideal in T) if and only if $A$
is piecewise $\mathcal{F}$-syndetic.
\end{theorem}

\begin{proof}
See \cite[Theorem 2.3]{key-4}.
\end{proof}
Our primary concern in this paper is to state and prove the new Central
Sets Theorem along filters. For that, we need to raise some prior
definitions and concepts.
\begin{definition}
Let $\left(S,\cdot\right)$ be a semigroup and $\mathcal{F}$ be a
filter on $S$. Then
  \begin{enumerate}
\item For any $l\in\mathbb{N}$, and any $l$ sequences $\langle x_{n}^{(i)}\rangle_{n=1}^{\infty}$
for $i\in\left\{ 1,2,\cdots,l\right\} $, define the zigzag finite
product 
\[
ZFP\left(\langle x_{n}^{(i)}\rangle_{i,n=1,1}^{l,\infty}\right)=\left\{ \begin{array}{c}
\prod_{t\in H}y_{t}:H\in\mathcal{P}_{f}\left(\mathbb{N}\right)\,\text{and}\\
y_{i}\in\left\{ x_{i}^{(1)},x_{i}^{(2)},\cdots,x_{i}^{(l)}\right\} \,\text{for\,any}\,i\in\mathbb{N}
\end{array}\right\} .
\]
\item for any $k\in \mathbb{N}$, define
\[
ZFP_k\left(\langle x_{n}^{(i)}\rangle_{i,n=1,1}^{l,\infty}\right)=\left\{ \begin{array}{c}
\prod_{t\in H}y_{t}:H\in\mathcal{P}_{f}\left(\mathbb{N}\right)\, \min H \geq k\, \text{and}\\
y_{i}\in\left\{ x_{i}^{(1)},x_{i}^{(2)},\cdots,x_{i}^{(l)}\right\} \,\text{for\,any}\,i\in\mathbb{N}
\end{array}\right\} .
\]
\item Let $G\in\mathcal{P}_{f}\left(^{\mathbb{N}}S\right)$, we will call
$G$ is $\mathcal{F}$-good if for
any $F\in\mathcal{F}$, there exists $k=k(F)\in \mathbb{N}$ such that  $ZFP_k\left(G\right)\subseteq F$ .
\item A set $B\subseteq S$ will be called a $\mathcal{F}$-$J$ set, if
for any finite collection of $\mathcal{F}$ good maps, say $F\in\mathcal{P}_{f}\left(^{\mathbb{N}}S\right)$,
there exist $a_{1},a_{2},\ldots,a_{m+1}\in S$ and $\left\{ h_{1},h_{2},\cdots,h_{m}\right\} \subset\mathbb{N}$
such that 
\[
a_{1}f\left(h_{1}\right)a_{2}f\left(h_{2}\right)\ldots a_{m}f\left(h_{m}\right)a_{m+1}\in B.
\]
\item $\mathcal{P}_{f}^{\mathcal{F}}\left(^{\mathbb{N}}S\right)=\left\{
F\in\mathcal{P}_{f}\left(^{\mathbb{N}}S\right): F \,\text{\,is}\,\mathcal{F}\,\text{\,good}\right\} .$
\item A set $A\subseteq S$ is called $\mathcal{F}$-central if and only
if there exists an idempotent ultrafilter $p\in K\left(\bar{\mathcal{F}}\right)$
such that $A\in p$.
   \end{enumerate}
\end{definition}

\noindent For every filter, one may not get $\mathcal{F}$-good maps. But such kind of maps exist in many filters. We postpone this until observation \ref{obs}. Now we can state our desired Central Sets Theorem along filters.
\begin{theorem}
\label{main} Let $\left(S,\cdot\right)$ be a semigroup and $\mathcal{F}$ be
a filter on $S$ such that $\bar{\mathcal{F}}$ is a semigroup. Let
$A$ be a $\mathcal{F}$ central set on $S$. Then there exist functions
$m:\mathcal{P}_{f}^{\mathcal{F}}\left(^{\mathbb{N}}S\right)\rightarrow\mathbb{N}$,
$\alpha\in\times_{F\in\mathcal{P}_{f}^{\mathcal{F}}\left(^{\mathbb{N}}S\right)}S^{m(F)+1}$
and $\tau\in\times_{F\in\mathcal{P}_{f}^{\mathcal{F}}\left(^{\mathbb{N}}S\right)}\mathcal{J}_{m(F)}$
such that

  \begin{enumerate}
\item if $F,G\in\mathcal{P}_{f}^{\mathcal{F}}\left(^{\mathbb{N}}S\right)$
and $G\subsetneq F$, then $\tau\left(G\right)\left(m\left(G\right)\right)$ $<$
$\tau\left(F\right)\left(1\right)$ and
\item when $n\in\mathbb{N}$, $G_{1},G_{2},\ldots,G_{n}\in\mathcal{P}_{f}^{\mathcal{F}}\left(^{\mathbb{N}}S\right)$,
$G_{1}\subsetneq G_{2}\subsetneq\ldots\subsetneq G_{n}$, and for
each $i\in\left\{ 1,2,\cdots,n\right\} $, $f_{i}\in G_{i}$, one
has 
\[
\prod_{i=1}^{n}x\left(m\left(G_{i}\right),\alpha\left(G_{i}\right),\tau\left(G_{i}\right),f_{i}\right)\in A.
\]
  \end{enumerate}
\end{theorem}  
In a recent work \cite{key-7}, the authors have characterized the closed subsemigroups and ideals of $\beta \mathcal{F}$. It would be interesting to see whether the above result can be improved in the way of \cite{key-7}.\\
\noindent {\bf Organisation of the paper:} The orientation of our paper is the following. In the next section,
we present the proof of the $\mathcal{F}$-Central Set Theorem. Then
in section 3, we deal with preservation of largeness of some well
known sets in combinatorics. The proofs in that section are both combinatorial
as well as algebraic. In this section, we will also provide examples of $\mathcal{F}$-good maps. In section 4, we provide some applications,
mainly from an algebraic and combinatorial point of view, of the previous theorems.

\section{The $\mathcal{F}$-Central Sets Theorem}
To proceed further let us recall the Hales-Jewett Theorem. Conventionally $[t]$ denotes the set $\{1,2,\ldots,t\}$ and words of length $N$ over the alphabet $[t]$ are the elements of $[t]^{N}.$ A variable word is a word over $[t]\cup\{*\}$ in where $*$ occurs at least once and $*$ denotes the variable. A combinatorial line is denoted by $L_{\tau}=\{\tau(1),\tau(2),\ldots,\tau(t)\}$ where
$\tau(*)$ is a variable word and $L_{\tau}$ is obtained by replacing the variable $*$ by $1,2,\ldots.t.$ The following theorem is due to Hales-Jewett.
\begin{theorem}[\cite{key-6-1}]
For all values $t,r\in\mathbb{N}$, there exists a number $HJ(r,t)$ such that, if $N\geq HJ(r,t)$ and $[t]^{N}$ is $r$ colored then there will exist a monochromatic combinatorial line.
\end{theorem}

To prove the $\mathcal{F}$-Central Sets Theorem we will need some lemmas.
\begin{lemma} \label{J} 
Let $A$ be a piecewise $\mathcal{F}$-syndetic set and $F\in\mathcal{P}_{f}^{\mathcal{F}}\left(^{\mathbb{N}}S\right)$ be $\mathcal{F}$- good. Then there exist $a_{1},a_{2},\ldots,a_{\ell+1}\in S$ and $\left\{ h_{1},h_{2},\ldots,h_{\ell}\right\} \subset\mathbb{N}$ such that, 
\[
a_{1}f\left(h_{1}\right)a_{2}f\left(h_{2}\right)\ldots a_{\ell}f\left(h_{\ell}\right)a_{\ell+1}\in A
\]
 for all $f\in F$.
\end{lemma}

\begin{proof}
Let $A$ be a piecewise $\mathcal{F}$- syndetic set. Let $V\in\mathcal{F},$ and then, by definition, there exists a finite set $F_{V}\subset V$ and $W_{F}\in\mathcal{F}$ such that, for all finite $H\subset W_{F}$, there exists $y\in V$ such that, $H\cdot y\subset F_{V}^{-1}A$. As $F$ is $\mathcal{F}$ good, choose $k=k(F)\in \mathbb{N}$. Let $r=\mid F_{V}\mid$. Consider the Hales-Jewett number $N=N\left(\vert F\vert,r\right)$. Then any $r$-partition of $F^{N}$gives a monochromatic combinatorial line. For any $\left(f_{1},f_{2},\ldots,f_{N}\right)\in F^{N}$, consider
the corresponding element 
\[
f_{1}\left(k+1\right)f_{2}\left(k+2\right)\cdots f_{N}\left(k+N\right).
\]

Let $H=\left\{ f_{1}\left(k+1\right)f_{2}\left(k+2\right)\cdots f_{N}\left(k+N\right):\left(f_{1},f_{2},\cdots,f_{N}\right)\in F^{N}\right\} \subseteq W_{V}$. Now consider each element $x$ in $H$ with the color $t\in F_{V}$ such that $txy\in A$, where $t$ is minimum. So this induces an $r$-coloring on $H$, and gives a monochromatic combinatorial line, i.e. there exists a finite set $H=\left\{ h_{1},h_{2},\cdots,h_{\ell}\right\} $ and elements $a_{1},a_{2},\cdots,a_{\ell+1}$ in $S$, such that 
\[
a_{1}f\left(h_{1}\right)a_{2}f\left(h_{2}\right)\ldots a_{\ell}f\left(h_{\ell}\right)a_{\ell+1}\in A
\]

for all $f\in F.$
\end{proof}
\noindent So, we have seen that every piecewise $\mathcal{F}$-syndetic set is $\mathcal{F}$-$J$ set. We will need an another lemma now.
\begin{lemma} \label{JJ}
Let $\left(S,\cdot\right)$ be a semigroup and $\mathcal{F}$
be a filter such that $\bar{\mathcal{F}}$ is a semigroup. Let $A$
be a $\mathcal{F}$-$J$ set and $F\in\mathcal{P}_{f}^{\mathcal{F}}\left(^{\mathbb{N}}S\right)$.
Then for all $m\in\mathbb{N}$, there exists $l\in\mathbb{N}$, such
that $a\in S^{l+1}$, $t\in\mathcal{J}_{l}$ and 
$$x\left(m,a,t,f\right)=a_{1}f\left(t\left(1\right)\right)a_{2}f\left(t\left(2\right)\right)\cdots a_{m}f\left(t\left(m\right)\right)a_{m+1}\in A$$
for all $f\in F$, where $t\left(1\right)>m$.
\end{lemma}

\begin{proof}
Take the function $g_{f}^{m}\left(x\right)=f\left(x+m\right)$ for all $x\in\mathbb{N}$. Let us consider $F_{m}=\left\{ g_{f}^{m}:f\in F\right\} .$ As $F$ is $\mathcal{F}$ good, $F_{m}$ is also $\mathcal{F}$-good.  So, choose $l\in\mathbb{N}$, $a\in S^{l+1},t\in\mathcal{J}_{l}$ such that for all $g_{f}^{m}\in F_{m}$ , $x\left(l,a,t,g_{f}^{m}\right)\in A$. Now this implies $x\left(l,a,t+m,f\right)\in A$ for all $f\in F$ , where $\left(t+m\right)\left(i\right)=t\left(i\right)+m$ for $i\in\left\{ 1,2,\cdots,l\right\} $ , and so $t+m\in\mathcal{J}_{l}$.  So rewriting $t+m=t'$ we have $x\left(l,a,t',f\right)\in A$ for all $f\in F$ , where $t'\left(1\right)>m$. 
So the lemma is proved.
\end{proof}
\noindent We are now in a position to prove $\mathcal{F}$-central set Theorem.

\vspace{0.15in}

\begin{proof}[\large{Proof of theorem \ref{main}}]
$\mathcal{F}$ be a filter on $S$ such that $\bar{\mathcal{F}}$ is a semigroup. Assume that $A$ be a $\mathcal{F}$-central set. Then there exists a minimal idempotent of $\bar{\mathcal{F}}$ say,  $p\in K\left(\bar{\mathcal{F}}\right)$ such that $A\in p$. As $p\cdot p=p$, so $A^{*}=\left\{ x:x^{-1}A\in p\right\} \in p$.
Now define $m\left(F\right)\in\mathbb{N}$, $\alpha\left(F\right)\in S^{m(F)+1}$
and $\tau\left(F\right)\in\mathcal{J}_{m(F)}$, by induction on $\mid F\mid$,
satisfying the following inductive hypothesis:
   \begin{enumerate}   
\item if $\emptyset\neq G\subset F$, then $\tau\left(G\right)\left(m\left(G\right)\right)<\tau\left(F\right)\left(1\right)$

\item if $n\in\mathbb{N}$, $\emptyset\neq G_{1}\subsetneq G_{2}\subsetneq\ldots\subsetneq G_{n}=F$
and $\langle f_{i}\rangle_{i=1}^{n}\in\times_{i=1}^{n}G_{i}$ then
    \end{enumerate}
\[
\prod_{i=1}^{n}x\left(m\left(G_{i}\right),\alpha\left(G_{i}\right),\tau\left(G_{i}\right),f_{i}\right)\in A^{*}.
\]

Assume first that $F=\left\{ f\right\} $. As $A^{*}$ is piecewise
$\mathcal{F}$- syndetic, pick $m\in\mathbb{N}$, $a\in S^{m+1}$,
$t\in\mathcal{J}_{m}$ such that $x\left(m,a,t,f\right)\in A^{*}$.
So the hypothesis is satisfied vacuously. Now assume $\mid F\mid>1$ and $m\left(G\right),\alpha\left(G\right),\tau\left(G\right)$
have been defined for all proper subsets $G$ of $F$.  Let $K=\left\{ \tau\left(G\right):\emptyset\neq G\subsetneq F\right\} $,
$m=\max_{\emptyset\neq G\subsetneq F}\tau\left(G\right)\left(m\left(G\right)\right)$.
Let 
\[
M=\left\{ \begin{array}{c}
\prod_{i=1}^{n}x\left(m\left(G_{i}\right),\alpha\left(G_{i}\right),\tau\left(G_{i}\right),f_{i}\right):n\in\mathbb{N},\emptyset\neq G_{1}\subsetneq\ldots\subsetneq G_{n}\subsetneq F,\\
\,\qquad\text{and}\,\langle f_{i}\rangle_{i=1}^{n}\in\times_{i=1}^{n}G_{i}
\end{array}\right\}. 
\]
Then $M$ is finite and $M\subset A^{*}$ by induction hypothesis.
Let 
\[
B=A^{*}\cap\bigcap_{z\in M}z^{-1}A^{*}.
\]
 Then $B\in p$ and by lemma \ref{J} and \ref{JJ} there exist $N\in\mathbb{N}$,
$a\in S^{N+1}$, $t\in\mathcal{J}_{N}$ and $t\left(1\right)>m$ such
that $x\left(N,a,t,f\right)\in B$ for all $f\in F$. 

We write $m\left(F\right)=N$, $\alpha\left(F\right)=a$, $\tau\left(F\right)=t$.
Then clearly $\tau\left(F\right)\left(1\right)=t\left(1\right)>\tau\left(G\right)\left(m\left(G\right)\right)$
for all $\emptyset\neq G\subsetneq F$. Denote $G_{n+1}=F$, so the first hypothesis of the induction is satisfied. 
Now let $\langle f_{i}\rangle_{i=1}^{n+1}\in\times_{i=1}^{n+1}G_{i}$.
Then $\prod_{i=1}^{n}x\left(m\left(G_{i}\right),\alpha\left(G_{i}\right),\tau\left(G_{i}\right),f_{i}\right)\in M$.
 So, from the above choice of $N$, $a$ and $t$ we have, 
 $$x\left(N,a,t,f_{n+1}\right)=x\left(m\left(F\right),\alpha\left(F\right),\tau\left(F\right),f_{n+1}\right)$$\\
$$\hspace{1.7in}\in \left(\prod_{i=1}^{n}x\left(m\left(G_{i}\right),\alpha\left(G_{i}\right),\tau\left(G_{i}\right),f_{i}\right)\right)^{-1}A^{*},$$
i.e. $\prod_{i=1}^{n+1}x\left(m\left(G_{i}\right),\alpha\left(G_{i}\right),\tau\left(G_{i}\right),f_{i}\right)\in A^{*}$.
Hence the second induction hypothesis is satisfied. So, we have the
desired Central Set Theorem along filters.
\end{proof}

\section{Preserving large sets}

In \cite{key-54.01}, N. Hindman and D. Strauss proved the following equivalent statements for a commutative group.
\begin{theorem}
Let $\left(S,+\right)$ be a commutative group and let $H$ be a subgroup
of $S$. Then the following statements are equivalent.
    \begin{enumerate}       
\item $H$ is IP{*} in $S.$

\item $H$ is central{*} in $S.$

\item $H$ is central in $S.$

\item $H$ is piecewise syndetic in $S.$
     \end{enumerate}
\end{theorem}

\begin{proof}
\cite[Theorem 1.16]{key-54.01}.
\end{proof}
The above proofs were purely algebraic in nature. Here we will show that, those groups are unexpectedly rich. Our proof will be combinatorial in nature.
\begin{theorem}\label{pwipr}
Let $(S,+)$ be a commutative semigroup containing the identity element
and let $G$ be a subgroup of $S$ which is a piecewise syndetic set
in $S$. Then there exists $r\in\mathbb{N}$ such that $G$ is $IP_{r}^{*}$
in $S$.
\end{theorem}

\begin{proof}
Since $G$ is piecewise syndetic, we can pick $H\in\mathcal{P}_{f}(S)$,
such that for every $F\in\mathcal{P}_{f}(S)$ there exists $y\in S$
such that $F+y\subseteq\bigcup_{t\in H}\left(-t+G\right)$. Let $r=\mid H\mid+1,$
and let a sequence $\langle x_{n}\rangle_{n=1}^{r}$in $S$ be given.
Let $F=\{\Sigma_{i=1}^{n}x_{i}:n\in\{1,2,\ldots,r\}\}$. We can pick
$y\in S$, such that $F+y\subseteq\bigcup_{t\in H}\left(-t+G\right)$.
For $n\in\left\{ 1,2,\ldots,r\right\} $, we pick $t_{n}\in H$ such
that $t_{n}+\Sigma_{i=1}^{n}x_{i}+y\in G$. By our choice of $r$
we can pick $m<n$ in $\left\{ 1,2,\ldots,r\right\} $ such that $t_{m}=t_{n}$. 
So we have $-\left(t_{m}+\Sigma_{i=1}^{m}x_{i}+y\right)\in G$ which
in turn implies,\\ 

$\hspace{0.5in}\quad -\left(t_{m}+\Sigma_{i=1}^{m}x_{i}+y\right)+\left(t_{n}+\Sigma_{i=1}^{n}x_{i}+y\right)$\\ 

$\hspace{0.5in} =-\left(t_{m}+\Sigma_{i=1}^{m}x_{i}+y\right)+\left(t_{m}+\Sigma_{i=1}^{m}x_{i}+y\right)+\Sigma_{i=m+1}^{n}x_{i} $\\

$\hspace{0.5in}=\sum_{i=m+1}^{n}x_{i}\in G.$

\noindent This implies $G$ is $IP_{r}^{*}$ in $S$.
\end{proof}
This result is also true for arbitrary group , but we have to slightly
modify the assumption.
\begin{theorem}
\label{non } Let $(S,\cdot)$ be an arbitrary group and let $G$ be a normal subgroup of $S$ which
is a piecewise syndetic set in $S$. Then there exists $r\in\mathbb{N}$
such that $G$ is $IP_{r}^{*}$ set in $S$.
\end{theorem}

\begin{proof}
Since $G$ is piecewise syndetic, pick $H\in\mathcal{P}_{f}(S)$ such
that for every $F\in\mathcal{P}_{f}(S)$, there exists $y\in S$ such
that $F\cdot y\subseteq\bigcup_{t\in H}t^{-1}G$. Let $r=\mid H\mid+1$,
and let a sequence $\langle x_{n}\rangle_{n=1}^{r}$in $S$ be given.
Let $F=\{\Pi_{i=1}^{n}x_{i}:n\in\{1,2,\ldots,r\}\}$. We can pick
$y\in S$ such that $F\cdot y\subseteq\bigcup_{t\in H}t^{-1}G$. For
$n\in\{1,2,\ldots,r\}$, we pick $t_{n}\in H$ such that $t_{n}\cdot\Pi_{i=1}^{n}x_{i}\cdot y\in G$.
By the similar argument of choice of proper $r$ for $\mid H\mid$
many coloring, we can again pick $m<n$ in $\{1,2,\ldots,r\}$ such
that $t_{m}=t_{n}$. 

So we get $\left(t_{m}\cdot\Pi_{i=1}^{m}x_{i}\cdot y\right){}^{-1}\left(t_{n}\cdot\Pi_{i=1}^{n}x_{i}\cdot y\right)\in G,$ $\text{i.e.}\,y^{-1}\cdot(\Pi_{i=1}^{m}x_{i})^{-1}\cdot t_{m}^{-1}\cdot t_{n}.\Pi_{i=1}^{n}x_{i}\cdot y=y^{-1}\cdot(\Pi_{i=1}^{m}x_{i})^{-1}\Pi_{i=1}^{n}x_{i}\cdot y\in G.$  Hence $\,\Pi_{i=m+1}^{n}x_{i}\in y\cdot G\cdot y^{-1}\subseteq G$
(Since $G$ is a normal subgroup).

So $G$ is $IP_{r}^{*}$ set in $S$.
\end{proof}
\begin{theorem}\label{jip}
Let $(S,+)$ be a commutative semigroup containing the identity element
and $(G,+)$ is a subgroup of $(S,+)$. If $G$ is a $J$-set in
$S$ then it is $IP^{*}$ set in $S$.
\end{theorem}

\begin{proof}
Let $\langle x_{n}\rangle_{n=1}^{\infty}$ be any sequence. We can consider this sequence as a function $f\in ^{\mathbb{N}}\!\!S $ such that $f(n)=x_n$ for each $n\in \mathbb{N}$. Since $G$ is a $J$-set, for $F=\left\lbrace f,\overrightarrow{0} \right\rbrace,$
where $\overrightarrow{0}=\left(0,0,\ldots\right),$ there exists $a\in S$ and
a $H\in\mathcal{P}_{f}(\mathbb{N}),$ such that for all $f\in F$ we
have $a+\Sigma_{t\in H}f(t)\in G$. So, $a\in G$ and $ a+\sum_{t\in H}x_t\in G.$ Hence it follows that 
$-a+\left(a+\Sigma_{t\in H}x_{t}\right)\in G$
(Since $G$ is a group) implying $\Sigma_{t\in H}x_{t}\in G$.
Hence $G$ is $IP^{*}$ set in $S$.
\end{proof}
\noindent So, if $\left(S,+\right)$ is a commutative semigroup containing the
identity element, then a subgroup $G$ of $S$ is a $J$-set implies
it is an IP$^{*}$set implies piecewise syndetic set implies an IP$_{r}^{*}$
set and so $J$-subgroups are IP$_{r}^{*}$ sets. If $S$ is non-commutative,
then a normal subgroup $G$ is a piecewise syndetic $\Rightarrow\text{IP}_{r}^{*}$.

In the next part of this section, we want to further generalize these
notions of largeness. We want the largeness along some filter, previously
mentioned in the introduction. For this generalization we need few
definitions a priory.
\begin{definition}
Let $\left(S,\cdot\right)$ be a semigroup, and $\mathcal{F}$
be a filter on $S$ such that $\bar{\mathcal{F}}$ is a semigroup.
Then define,
   \begin{enumerate}
\item A set $A\subset S$ will be called $\mathcal{F}$-IP if and only if
$A\in p$ for some $p\in E\left(\bar{\mathcal{F}}\right)=\left\{ q\in\bar{\mathcal{F}}:q\,\text{is\,an\,idempotent}\right\} $.
\item A set $C$ is an $\mathcal{F}$-$IP${*} set if it intersects every
$\mathcal{F}$-$IP$ set.
\item Let $\langle x_{n}\rangle_{n=1}^{\infty}$ be any sequence. We call
$\langle FP\langle x_{n}\rangle_{n=1}^{\infty}\rangle$ an $\mathcal{F}_{\subseteq}$-$IP$
sequence if for any $F\in\mathcal{F}$, there exist $m=m\left(F\right)\in\mathbb{N}$
such that $FP\langle x_{n}\rangle_{n=m}^{\infty}\subseteq F$. For existence of such sequences and filters see theorem \ref{existence}.
\item Similarly, a set $A$ which contains $FP\left(\langle x_{n}\rangle_{n=m}^{m+r}\right)$
for some $m\in \mathbb{N}$ and $\mathcal{F}_{\subseteq}$-$IP$ sequence $\langle x_{n}\rangle_{n=1}^{\infty}$ in $S$, is called an $\mathcal{F}_{\subseteq}$-IP$_{r}$ set.
\item For $r\in \mathbb{N}$, A set $C$ is called an $\mathcal{F}_{\subseteq}$-IP$_{r}${*}
set, if it intersects every $\mathcal{F}_{\subseteq}$-IP$_{r}$ set. That means, like $IP_{r}^{*}$ sets, this 
type of sets contains elements of the form $x_{i_{1}}\cdot x_{i_{2}}\cdots x_{i_{l}}$ where $1 \leq l \leq r$ . Clearly this is an stronger set than $\mathcal{F}$-$IP$ sets. If we take the filter $\mathcal{F}=\left\lbrace S \right\rbrace$ then it is an $IP_{r}^*$ set.
   \end{enumerate}
\end{definition}
Here we provide an important theorem regarding the existence of $\mathcal{F}_{\subseteq}$-$IP$
sequences.
\begin{theorem}\label{existence}
For any semigroup $\left(S,\cdot\right)$, there exists filters which
generates a closed subsemigroup as well as contain $\mathcal{F}_{\subseteq}$-$IP$
sequences.
\end{theorem}

\begin{proof}
Now we take any IP sequence $I=\langle FP\langle x_{n}\rangle_{n=1}^{\infty}\rangle$.
Consider the family of sets: 
\[
\mathcal{A}=\left\{ B\subseteq I:\, \mid I\setminus B\mid<\infty\right\} 
\]
Now generate a filter $\mathcal{F}$ by $\mathcal{A}$ as: 
\[
\mathcal{F}=\left\{ F:F\supseteq B\,\text{for\,some}\,B\in\mathcal{A}\right\} .
\]

\noindent Now our claim is that $\mathcal{F}$ is an idempotent filter.

Let $C\in\mathcal{F},$ so $FP\langle x_{n}\rangle_{n=m}^{\infty}\subseteq C$
for some $m\in\mathbb{N}$. Then for each $y\in FP\langle x_{n}\rangle_{n=m}^{\infty}$,
\[
y^{-1}\left(FP\langle x_{n}\rangle_{n=m}^{\infty}\right)\supseteq FP\langle x_{n}\rangle_{n=N}^{\infty},
\]
 for some $N\in\mathbb{N}.$ So, for all $y\in FP\langle x_{n}\rangle_{n=m}^{\infty}$,
\[
y^{-1}C\supseteq y^{-1}\left(FP\langle x_{n}\rangle_{n=m}^{\infty}\right)\supseteq FP\langle x_{n}\rangle_{n=N}^{\infty}\in\mathcal{A}.
\]

\noindent Hence, $y^{-1}C\in\mathcal{F}.$ So, $\left\{ y:y^{-1}C\in\mathcal{F}\right\} \supseteq FP\langle x_{n}\rangle_{n=m}^{\infty}\in\mathcal{A}.$
Hence we get that $C\in\mathcal{F}$ implies $\left\{ y:y^{-1}C\in\mathcal{F}\right\} \in\mathcal{F}.$
So $\mathcal{F}\subseteq\mathcal{F}\cdot\mathcal{F}$ and hence $\mathcal{F}$
is an idempotent filter. So we get $\langle FP\langle x_{n}\rangle_{n=1}^{\infty}\rangle$
is an $\mathcal{F}_{\subseteq}$-$IP$ sequence and $\overline{\mathcal{F}}$
is a closed subsemigroup.
\end{proof}
\noindent From above, here we summarise two observations.
\begin{observation}\label{obs}
Let $(S,\cdot)$ be a semigroup. Then,\\
   \begin{itemize}
\item Suppose we have taken an IP-set $FP\left(x_{n}\right)_{n=1}^{\infty}$
such that
\[ 
\bigcap_{m=1}^{\infty}\overline{FS\left(x_{n}\right)_{n=m}^{\infty}}\bigcap cl\left(K\left(\beta S\right)\right)=\emptyset.
\]
Then from this IP-set, we can generate an filter and thus we obtain a closed subsemigroup which does not intersect $K(\beta S)$. The central sets in this subsemigroup are of course not central sets of $S$. This shows the originality of $\mathcal{F}$ Central Sets Theorem.
\item Suppose $\langle FP\langle x_{n}\rangle_{n=1}^{\infty}\rangle$ is
a $\mathcal{F}_{\subseteq}$-$IP$ sequence in a semigroup $\left(S,\cdot\right)$.
Now we can construct a $\mathcal{F}$-good map as: $F=\left\{ \left\{ x_{2n+1}\right\} _{n=1}^{\infty},\left\{ x_{2n}\right\} _{n=1}^{\infty}\right\} $
or more generally for any $n\in\mathbb{N}$, 
\[
F=\left\{ \left\{ x_{nm}\right\} _{m=1}^{\infty},\left\{ x_{nm+1}\right\} _{m=1}^{\infty},\ldots,\left\{ x_{nm+\left(n-1\right)}\right\} _{m=1}^{\infty}\right\}.
\]
\end{itemize}
\end{observation}
\noindent From now without stated, we will assume those cases where  $\mathcal{F}_{\subseteq}-IP$ sequences exists, when it will be necessary.
 Now we are in a position to prove the largeness along filters. The
following proof is similar to \cite[Theorem 1.16]{key-54.01}.
\begin{theorem}
\label{fip} Let $\left(S,\cdot\right)$ be a commutative semigroup containing the identity element and $\mathcal{F}$
be a filter on $S$ such that $\bar{\mathcal{F}}$ is a semigroup.
Then if $\left(H,\cdot\right)$ is a subgroup of $S$ which is piecewise
$\mathcal{F}$-syndetic then $H$ is $\mathcal{F}$-$IP${*}.
\end{theorem}

\begin{proof}
Let $T=\overline{\mathcal{F}}$. Let $H$ be piecewise $\mathcal{F}$-syndetic, so we can pick $p\in\overline{H}\cap K\left(T \right)$.
Let $L$ be a minimal left ideal of $T$ such that
$p\in L$. Let $s\in E\left(T \right)$. Choose a minimal
left ideal $M$ of $T$ with $M\subseteq T\cdot s$
and $r\in E\left(M\right)$. Let $R=r\cdot T$, so, $R$ is minimal
right ideal of $T$ by \cite[Theorem 1.59]{key-53}. Hence $R\cap L$
is a group. Let $q$ be the identity of $R\cap L$. Since $p\in L=T\cdot q$
by \cite[Lemma 1.30]{key-53}, $p\cdot q=p$. Therefore $H\in p\cdot q$
and so, $\left\{ x\in S:x^{-1}H\in q\right\} \in p$. So pick $x\in H$
such that $x^{-1}H\in q$, i.e. $H\in q$.

Now $q\in r\cdot T$ and so, $r\cdot q=q$. Therefore $\left\{ x\in S:x^{-1}H\in q\right\} \in r$.
But then $\left\{ x\in S:x^{-1}H\cap H\in q\right\} \in r$ as $H\in q$.
But as we know $H$ is a group, we have $\left\{ x\in S:x^{-1}H\cap H\in q\right\} \subseteq H$.
So, $H\in r$. Since $r\in T\cdot S$, $r\cdot s=r$. So $\left\{ x\in S:x^{-1}H\in s\right\} \in r$
and hence $H\in s$ as $H\in r$. As $s\in E\left(\bar{\mathcal{F}}\right)$
is arbitrary, $H$ is $\mathcal{F}$-$IP${*} set.
\end{proof}
\begin{theorem}\label{fdip}
Let $(S,+)$ be a commutative semigroup containing the identity element
and $G$ be a piecewise $\mathcal{F}$-syndetic subgroup of $S$.
Then $G$ is an $\mathcal{F_{\subseteq}}$-$IP_{r}${*} set.
\end{theorem}

\begin{proof}
Since $G$ is piecewise $\mathcal{F}$-syndetic, we can fix $V\in\mathcal{F}$,
so that there is a finite $F_{V}\subseteq V$ (let $\mid F_{V}\mid=r)$
and $W_{V}\in\mathcal{F}$ such that for any finite $H\subseteq W_{V},$
there is $x\in V,$ which altogether satisfies $H+x\subseteq-F_{V}+G$.
Let $\langle x_{n}\rangle_{n=1}^{\infty}$ be a dominated $\mathcal{F}\,\text{-}\,$ IP
sequence in $S$ and let $k=k(W_V)\in \mathbb{N}$ such that  $FS\langle x_{n}\rangle_{n=k}^{\infty}\subset W_{V}\in\mathcal{F}$.
Let $H=\left\{ \sum_{n=k}^{j}x_{n}:k\leq j\leq r+k\right\} $,
then there exists $x\in V$ such that, $H+x\subseteq-F_{V}+G$. By
the same coloring argument we are using, there exists $m<l$
where $m,l\in\left\{ k,k+1,\cdots,k+r\right\} ,$ and a $t\in F_{V}$
such that $t+\sum_{n=k}^{m}x_{n}+x\in G$ , and $t+\sum_{n=k}^{l}x_{n}+x\in G$.
Since $G$ is a group, we get $\sum_{n=m+1}^{l}x_{n}\in G$ , and
so, $\left(G,+\right)$ is an $\mathcal{F_{\subseteq}}$-$IP_{r}${*}
set.
\end{proof}
Similarly proceeding as in proof of theorem \ref{non } one can easily
derive the following
\begin{theorem}
Let $(S,\cdot)$ be an arbitrary group  and let $G$ be a normal subgroup of $S$ which is a piecewise
$\mathcal{F}$-syndetic set in $S$. Then there exists $r\in\mathbb{N}$
such that $G$ is $\mathcal{F_{\subseteq}}$-$IP_{r}^{*}$ set.
\end{theorem}

\begin{proof}
Left to the reader.
\end{proof}
An analogue version of theorem \ref{jip} is the following.
\begin{theorem}\label{fjdip}
Let $\left(S,+\right)$ be a commutative semigroup with identity and
$G$ be a subgroup of $S$. Let $\mathcal{F}$ be a filter on $G$
such that $\bar{\mathcal{F}}$ is a semigroup and $0\in F$ for all
$F\in\mathcal{F}.$ Then $G\subseteq S$ is an $\mathcal{F}$-$J$
group implies it is $\mathcal{F}_{\subseteq}$-$IP${*} in $S$.
\end{theorem}
\begin{proof}
Let $\langle x_{n}\rangle_{n=1}^{\infty}$ be any $\mathcal{F}_{\subseteq}$-$IP$
sequence. Let $F=\langle x_{n}\rangle_{n=1}^{\infty}\cup\vec{0}$.
Then $F\in\mathcal{P}_{f}^{\mathcal{F}}\left(^{\mathbb{N}}S\right)$.
Choose $a\in G$, $H\in\mathcal{P}_{f}\left(\mathbb{N}\right)$ such
that $a+\sum_{t\in H}f\left(t\right)\in G$ for all $f\in F$. Since
$G$ is a group and taking $f\left(t\right)=x_{t}$ we get $\sum_{t\in H}x_{t}\in G$.
Hence $G$ is an $\mathcal{F}_{\subseteq}$-$IP${*} in $S$.
\end{proof}

\section{Applications}

Here in this section we look forward to some applications of the theorems
proved above. Denote by $\left(G_1,\mathcal{F}\right)$ a group
with a corresponding filter. Let $\varphi:\left(G_1,\mathcal{F}\right)\rightarrow\left(G_2,\mathcal{G}\right)$ be a group homomorphism with corresponding filters. Let us introduce
a new notion of homomorphism here in this article as following.
\begin{definition}
Let $G_{1}$ and $G_{2}$ be two groups, and $\mathcal{F}$ and $\mathcal{G}$
 are two filters on these two groups respectively. Let $\varphi:G_{1}\rightarrow G_{2}$
be a group homomorphism such that for all $F\in\mathcal{F}$, there
exists $G\in\mathcal{G}$ satisfying $F=\varphi^{-1}\left(G\right)=\left\{ x\in G_1:\varphi(x)\in G\right\} $. We call such $\varphi$ a $\left(\mathcal{F},\mathcal{G}\right)$ good homomorphism.
\end{definition}
Note that for any arbitrary group $G_{1}$ and $G_{2}$, if we consider
the filter $\mathcal{F}=\left\{ G_{1}\right\} $, and $\mathcal{G}=\left\{ G_{2}\right\} $
then any homomorphism from $G_{1}$ to $G_{2}$ is a $\left(\mathcal{F},\mathcal{G}\right)$
good homomorphism. We know that any homomorphism $\varphi:G_{1}\rightarrow G_{2}$
extends continuously to a homomorphism $\Phi:\beta G_{1}\rightarrow\beta G_{2}$.
Note for any $p\in\beta G_{1}$, $\Phi(p)=\left\{ B\subseteq G_{2}:\varphi^{-1}(B)\in p\right\} .$
\begin{lemma}
If $\varphi:G_{1}\rightarrow G_{2}$ be a $\left(\mathcal{F},\mathcal{G}\right)$
good group homomorphism then $\phi^{-1}(\overline{\mathcal{G}})\subseteq\overline{\mathcal{F}}$.
\end{lemma}

\begin{proof}
Let $p\in\Phi^{-1}(\overline{\mathcal{G}})$ implies $\Phi\left(p\right)\in\overline{\mathcal{G}}$.
Now $\Phi\left(p\right)\in\overline{\mathcal{G}}$ implies $G\in\Phi\left(p\right)$
for all $G\in\mathcal{G}$. Let $F\in\mathcal{F}$ and choose $G_{F}\in\mathcal{G}$
such that $F=\varphi^{-1}(G_{F})$. So $G_{F}\in\Phi\left(p\right)$
and this implies $\varphi^{-1}(G_{F})=F\in p$. So $F\in\mathcal{F}$
implies $F\in p$, which implies that $p\in\bigcap_{F\in\mathcal{F}}\bar{F}$,
i.e $p\in\overline{\mathcal{F}}$ and we are done.
\end{proof}
In \cite{key-31.1} the authors proved the following theorem combinatorially.
\begin{theorem}
Let $\varphi:\left(S,\cdot\right)\rightarrow\left(T,\cdot\right)$
be a semigroup homomorphism. Then,
\begin{enumerate}
\item If $\varphi\left(S\right)$ is piecewise syndetic in $T$, then
$A\subseteq S$ is piecewise syndetic/central-set implies $\varphi\left(A\right)$
is also a piecewise syndetic/central set respectively in $T$.
\item If $\varphi\left(S\right)$ is $J$-set in $T$, then $A\subseteq S$
is a $J$-set implies $\varphi\left(A\right)$ is also a $J$-set
respectively in $T$.
\end{enumerate}
\end{theorem}

In \cite{key-5}, the corresponding result on piecewise syndetic sets
and central sets were proved algebraically. Here from the above theorems and theorems \ref{pwipr}, \ref{fip}
we get that, if $G_1$ and $G_2$ are abelian groups, and $\varphi:\left(G_{1},\cdot\right)\rightarrow\left(G_{2},\cdot\right)$
be a group homomorphism, and $\varphi\left(G_{1}\right)$ be piecewise
syndetic/$J$-set in $G_{2}$, then if $H$ is a subgroup of $G_{1}$
which is a piecewise syndetic/$J$-set in $G_{1}$ implies
$\varphi\left(H\right)$ is an IP$_{r}^{*}$/ IP$^{*}$ set in $G_{2}$
respectively.Now we have the following results.
\begin{theorem} \label{fspreserve}
Assume $G_{1}$and $G_{2}$be two groups and $\mathcal{F}$ and $\mathcal{G}$
be corresponding filters on them such that $\overline{\mathcal{F}} $ and $\overline{\mathcal{G}} $ are semigroups. Let $\varphi:\left(G_{1},\mathcal{F}\right)\rightarrow\left(G_{2},\mathcal{G}\right)$
be a $(\mathcal{F},\mathcal{G})$ good homomorphism such that $\varphi(G_{1})$ is piecewise $\mathcal{G}$-syndetic.  Then for any piecewise $\mathcal{F}$-syndetic subgroup $H_{1}$ of $G_{1}$we have
$\varphi(H_{1})$ is piecewise $\mathcal{G}$-syndetic.
\end{theorem}

\begin{proof}
By definition $\varphi(G_{1})$ is piecewise $\mathcal{G}$-syndetic,
so $\overline{\varphi(G_{1})}\cap K\left(\overline{\mathcal{G}}\right)\neq\emptyset$.
This implies that $\Phi(\overline{G_{1}})\cap K\left(\overline{\mathcal{G}}\right)\neq\emptyset.$
Hence $\Phi(\beta G_{1})\cap K\left(\overline{\mathcal{G}}\right)\neq\emptyset$
and so $\Phi^{-1}\left(K\left(\overline{\mathcal{G}}\right)\right)\neq\emptyset$.
As $K\left(\overline{\mathcal{G}}\right)$ is a two sided ideal, we
have $\Phi^{-1}\left(K\left(\overline{\mathcal{G}}\right)\right)$ is also a two sided ideal.
As $\Phi^{-1}\left(\overline{\mathcal{G}}\right)\subseteq\overline{\mathcal{F}}$, we have $K\left(\overline{\mathcal{F}}\right)\subseteq\Phi^{-1} \left( K \left(\overline{\mathcal{G}}\right)\right)$.
By our hypothesis $H_{1}$ is piecewise $\mathcal{F}$-syndetic subgroup
so $\overline{H_{1}}\cap K\left(\overline{\mathcal{F}}\right)\neq\emptyset$,
i.e. $\Phi\left(\overline{H_{1}}\right)\cap\Phi\left(K\left(\overline{\mathcal{F}}\right)\right)\neq\emptyset$,
i.e. $\overline{\varphi\left(H_{1}\right)}\cap K\left(\overline{\mathcal{G}}\right)\neq\emptyset.$
Hence $\varphi(H_{1})$ is piecewise $\mathcal{G}$-syndetic.
\end{proof}
The following one is the corresponding version of theorem \ref{fspreserve} for $J$-sets.

\begin{theorem}\label{fjpreserve}
Let $\left(G_{1},\cdot\right)$ and $\left(G_{2},\cdot\right)$ be
two groups with respective filters $\mathcal{F}$and $\mathcal{G}$
such that $\bar{\mathcal{F}}$and $\bar{\mathcal{G}}$ are semigroups.
Let $\varphi:\left(G_{1},\cdot\right)\rightarrow\left(G_{2},\cdot\right)$
be a $\left(\mathcal{F},\mathcal{G}\right)$ good homomorphism such
that $\varphi\left(G_{1}\right)$ is a $\mathcal{G}$-$J$ set in
$G_{2}$. Then for any $\mathcal{F}\text{-}J$ set $A\subseteq G_{1}$,
$\varphi\left(A\right)$ is a $\mathcal{G}\text{-}J$ set in $G_{2}$.
\end{theorem}
\begin{proof}
Given $\varphi\left(G_{1}\right)$ is a $\mathcal{G\,}$-$\,J$ set
in $G_{2}$, let $F\in\mathcal{P}_{f}^{\mathcal{G}}\left(^{\mathbb{N}}G_{2}\right)$
and $m\in\mathbb{N}$, $a\in G_{2}^{m+1}$, $t\in\mathcal{J}_{m}$
such that for all $f\in F$, $x\left(m,a,t,f\right)\in\varphi\left(G_{1}\right).$
Let $t\,'=\text{max}\left\{ t\left(1\right),t\left(2\right),\ldots,t\left(m\right)\right\} =t\left(m\right)$,
since $t$ is increasing.

Now for each $f\in F$, define $g_{f}\in^{\mathbb{N}}G_{2}$ such
that $g_{f}\left(n\right)=f\left(t\,'+n\right)$ for all $n\in\mathbb{N}$.
Take $G\in\mathcal{P}_{f}\left(^{\mathbb{N}}H\right)$ such that $G=\left\{ g_{f}:f\in F\right\} $.
As $F\in\mathcal{P}_{f}^{\mathcal{G}}\left(^{\mathbb{N}}G_{2}\right)$
so is $G\in\mathcal{P}_{f}^{\mathcal{G}}\left(^{\mathbb{N}}G_{2}\right)$.

Again since $\varphi\left(G_{1}\right)$ is a $\mathcal{G}$-$J$
set and $G\in\mathcal{P}_{f}^{\mathcal{G}}\left(^{\mathbb{N}}G_{2}\right)$,
we can again apply the above argument. Using the above argument repeatedly
we obtain sequences $\left\{ m_{n}\right\} _{n=1}^{\infty}$, $\left\{ a_{n}\right\} _{n=1}^{\infty}$
and $\left\{ t_{n}\right\} _{n=1}^{\infty}$ with $m_{n}\in\mathbb{N}$
, $a_{n}\in G_{2}^{m_{n}+1}$ and $t_{n}\in\mathcal{J}_{m}$ for all
$n\in\mathbb{N}$ such that for all $f\in F$, $x\left(m_{n},a_{n},t_{n},f\right)\in\varphi\left(G_{1}\right)$
and $\text{max}\left\{ t_{n}\right\} <\text{min}\left\{ t_{n+1}\right\} $.

Define $F'\in\mathcal{P}_{f}\left(^{\mathbb{N}}\varphi\left(G_{1}\right)\right)$
as $F'=\left\{ \left\{ x\left(m_{n},a_{n},t_{n},f\right)\right\} _{n=1}^{\infty}:f\in F\right\} $.
As $F\in\mathcal{P}_{f}^{\mathcal{G}}\left(^{\mathbb{N}}G_{2}\right)$
we have $F'\in\mathcal{P}_{f}^{\mathcal{G}}\left(^{\mathbb{N}}G_{2}\right)$.
For each $f'\in F'$ define $f''\in F''$ where $F''\in\mathcal{P}_{f}\left(^{\mathbb{N}}G_{1}\right)$
such that $F''=\left\{ f'':\varphi\left(f''\left(n\right)\right)=f'\left(n\right)\right\} $.
As $F'\in\mathcal{P}_{f}^{\mathcal{G}}\left(^{\mathbb{N}}\left( \varphi\left(G_{1}\right) \right)\right)$
and $\varphi$ is a $\left(\mathcal{F},\mathcal{G}\right)$ good homomorphism,
one can easily verify that $F''\in\mathcal{P}_{f}^{\mathcal{F}}\left(^{\mathbb{N}}G_{1}\right)$.

Now, $A\subseteq G_{1}$ is a $\mathcal{F}\text{-}J$ set, so there
exist $p\in\mathbb{N}$, $a\in G_{1}^{p+1}$, $t\,'\in\mathcal{J}_{p}$
such that $x\left(p,a,t',f''\right)\in A$ for all $f''\in F''$,
i.e. $\varphi\left(x\left(p,a,t',f''\right)\right)\in\varphi\left(A\right)$
for all $f''\in F''$, i.e. $x\left(p',a',t'',f'\right)\in\varphi\left(A\right)$
for all $f'\in F'$ where $a'\in T^{p'+1}$, $t''\in\mathcal{J}_{p'}$
and one can explicitly compute $p'$, $a'$, $t''$ by expanding
$x\left(p,a,t',f''\right)$ and then applying $\varphi$ on it.

As, $G$ contains the identity element, one can put $a(j)=e,$ the identity, for some $j\in \mathbb{N}$, if necessary, to conclude that  there exist $r\in\mathbb{N}$, $a\in G_{2}^{r+1}$,
$t\in\mathcal{J}_{r}$ such that, $x\left(p',a',t'',f'\right)=x\left(r,a,t,f\right)$
for all $f\in F$. Hence $\varphi\left(A\right)$ is a $\mathcal{G}\text{-}J$
set in $G_{2}$.
\end{proof}
\noindent Let us now summarize the following observations which we can obtain from the above theorems.
\begin{theorem}
Let $G_{1}$ and $G_{2}$ be two groups and $\mathcal{F}$ and $\mathcal{G}$
 are two filters on these two groups respectively. Let $\varphi$ be a $\left(\mathcal{F},\mathcal{G}\right)$ good homomorphism between these two groups. Then, 
   \begin{enumerate}
\item If $\varphi(G_{1})$ is piecewise $\mathcal{G}$-syndetic, then for any piecewise $\mathcal{F}$-syndetic subgroup $H_1$ of $G_1$, we have $\varphi(H_{1})$ is $\mathcal{G}$-$IP${*}
$G_{2}$.
\item If $G_1 $ and $G_2$ are both commutative groups, then under the above assumptions,
$\varphi(H_{1})$ is $\mathcal{G}_{\subseteq}$-$IP_{r}^{*}$ set in $G_{2}$.
\item Let $\varphi:\left(G_{1},\mathcal{F}\right)\rightarrow\left(G_{2},\mathcal{G}\right)$
be a $\left(\mathcal{F},\mathcal{G}\right)$ good homomorphism between
two abelian groups. Then if $\varphi\left(G_{1}\right)$ is a $\mathcal{G}$-$J$
set in $G_{2}$ then for any subgroup $H_{1}$ of $G_{1}$, which
is a $\mathcal{F}$-$J$ set, we have that $\varphi\left(H_{1}\right)$
is a $\mathcal{G}_{\subseteq}$-$IP${*} set in $G_{2}$.
    \end{enumerate}
\end{theorem}
\begin{proof}
For 1, use theorem \ref{fip} and theorem \ref{fspreserve}, for 2, use theorem \ref{fdip} and theorem \ref{fspreserve}, and for 3, use theorem \ref{fjdip} and theorem \ref{fjpreserve}. 
\end{proof}
As a trivial application we see that, if we take two groups $\left(G,\cdot \right)$ and $\left(H,\cdot \right)$ and the corresponding filters $\mathcal{F}=\left\lbrace G\right\rbrace$ and $\mathcal{H}=\left\lbrace H\right\rbrace$, then any group homomorphism $\varphi:G\rightarrow H$ is a $\left(\mathcal{G},\mathcal{H}\right)$ good homomorphism. So, if we assume $\varphi\left(G\right)$ is piecewise syndetic in $H$ then for any subgroup $K$ of $G$, which is piecewise syndetic in $G$, $\varphi\left(K\right)$ is piecewise syndetic in $H$ and so is an IP$^*$ set in $H$.
Again if we consider those two groups $G$ and $H$ are commutative, then $\varphi\left(K\right)$ is an IP$_{r}^*$ set in $H$.\\
\vspace{0.2in}

\noindent \textbf{Acknowledgment:} Both the authors are thankful to their supervisor
Prof. Dibyendu De for his guidance and continuous support. The first
author acknowledges the Grant UGC-NET SRF fellowship with ID No. 421333
of CSIR-UGC NET December 2016 and the second author acknowledges the
Grant CSIR-NET SRF fellowship with file No. 09/106(0184)/2019-EMR-I
of CSIR-UGC NET.
$\vspace{0.1in}$

\end{document}